\documentclass[12pt,a4paper]{article}
\usepackage{a4,a4wide}
\usepackage{amsfonts,amsmath,amssymb,amsthm}

\usepackage{color}
\usepackage{ifpdf}
\ifpdf
    \usepackage[pdftex]{graphicx}
    \DeclareGraphicsExtensions{.png,.pdf,.jpg}
\else
   \usepackage[dvips]{graphicx}
   \DeclareGraphicsExtensions{.eps}
\fi
\graphicspath{{.}{figuresPL/}}
\usepackage[latin1]{inputenc}
\numberwithin{equation}{section}

\newtheorem{theorem}{Theorem}[section]
\newtheorem{lemma}[theorem]{Lemma}
\newtheorem{proposition}[theorem]{Proposition}

\newtheorem{remark}[theorem]{Remark}

\renewcommand\tilde{\widetilde}

\def\R{\mathbb{R}}

\def\E{\mathcal{E}}

\def\H{\mathcal{H}}

\def\LM#1{\hbox{\vrule width.2pt \vbox to#1pt{\vfill \hrule width#1pt
height.2pt}}}
\def\LL{{\mathchoice {\>\LM7\>}{\>\LM7\>}{\,\LM5\,}{\,\LM{3.35}\,}}}
\def\restr{{\LL}}
\renewcommand{\phi}{\varphi}

\def\1{\mathbf{1}}

\def\Xint#1{\mathchoice
{\XXint\displaystyle\textstyle{#1}}%
{\XXint\textstyle\scriptstyle{#1}}%
{\XXint\scriptstyle\scriptscriptstyle{#1}}%
{\XXint\scriptscriptstyle\scriptscriptstyle{#1}}%
\!\int}
\def\XXint#1#2#3{{\setbox0=\hbox{$#1{#2#3}{\int}$ }
\vcenter{\hbox{$#2#3$ }}\kern-.57\wd0}}

\def\dashint{\Xint-}

\def\eps{\varepsilon}
\renewcommand{\subset}{\subseteq}

\def\lt{\left}
\def\rt{\right}
\def\les{\lesssim}
\def\ges{\gtrsim}

\def\E{\mathcal{E}}

\begin{document}
\title{An $\eps-$regularity result for  optimal transport maps between continuous densities}
\author{M. Goldman
\thanks{Universit\'e de Paris, CNRS, Sorbonne-Universit\'e,  Laboratoire Jacques-Louis Lions (LJLL), F-75005 Paris, 
France, \texttt{michael.goldman@u-paris.fr}}} 
\maketitle
\begin{abstract}\noindent
The aim of this short note is to extend the recent variational proof of partial regularity for optimal transport maps to the case of continuous densities.
\end{abstract}

\section{Introduction}
The aim of this short note is to extend the partial regularity result for optimal transport maps  obtained in \cite{GO} to the case of continuous densities (rather than H\"older continuous).  
 The interest lies in the proof rather than in the result in itself since it is known to hold under the weaker assumption that the densities are bounded from above and below \cite{FigKim}.
  Indeed, we show that for the squared Euclidean cost, both the variational approach to
  regularity theory for the Monge-Amp\`ere equation recently developed in \cite{GO,GHO,OM} and the one of \cite{DePFig} lead to the same result.
 We must however emphasize that the major achievement of \cite{DePFig} is the treatment of arbitrary cost functions (see \cite{OPR} for the extension of the variational approach to that setting).
 Our main $\eps-$regularity theorem is the following (compare with \cite[Th. 4.3]{DePFig}):
\begin{theorem}\label{theoepsintro}
 Let $\rho_0$ and $\rho_1$ be  densities with compact support\footnote{we assume compactness of the supports for simplicity. The statement is valid as soon as an optimal transport map exists (see \cite{Viltop}).} and  $T$ be the optimal transport map from $\rho_0$ to $\rho_1$ for the squared Euclidean cost on $\R^d$. For every $\alpha\in (0,1)$, there exists $\eps(\alpha,d)>0$ such that if  for some $R>0$,
 \begin{equation*}
 \frac{1}{(2R)^{d+2}}\int_{B_{2R}} |T-x|^2\rho_0+\|1-\rho_0\|_{L^\infty(B_{2R})}^2+\|1-\rho_1\|_{L^\infty(B_{2R})}^2\le \eps, 
 \end{equation*}
 then $T$ is of class $C^{0,\alpha}$ in $B_{R}$.
\end{theorem}

With this $\eps$-regularity result at hand and arguing as for \cite[Th. 1.1]{GO}, it is not hard to prove that  $T$ is a $C^{0,\alpha}$ homeomorphism outside of a set of measure zero if $\rho_0$ and $\rho_1$ are continuous.
\begin{theorem}\label{theo:final}
 For $E$ and $F$ two bounded open sets, let $\rho_0:E\to \R^+$ and $\rho_1:F\to \R^+$ be two continuous  densities with equal masses, both bounded  and bounded away from zero and let $T$ be the optimal transport map between $\rho_0$ and $\rho_1$. 
Then, there exist
 open sets $E'\subset E $ and $F'\subset F$
 with $|E\backslash E'|=|F\backslash F'|=0$ and such that  for every $\alpha\in (0,1)$, $T$ is a $C^{0,\alpha}$ homeomorphism between $E'$ and $F'$.
\end{theorem}

The proof of Theorem \ref{theoepsintro} follows very closely the proof of \cite[Th. 1.2]{GO}. 
It is based on a Campanato iteration scheme which uses at his heart an harmonic approximation result (see Proposition \ref{prop:Lagestim}). The main difference with \cite[Th. 1.2]{GO} 
lies in the iteration argument (see Theorem \ref{P3} below). Indeed, for continuous densities
the linear part of the affine transformations introduced in the excess improvement by tilting estimate do not necessarily converge to the identity. This causes the possible blow-up of the $C^{1,\alpha}-$norms.

\section{Notation}\label{sec:not}
In the paper we will use the following notation. The symbols $\sim$, $\ges$, $\les$ indicate estimates that hold up to a global constant $C$,
which typically only depends on the dimension $d$ and the H\"older exponent $\alpha$ (if applicable). 
For instance, $f\les g$ means that there exists such a constant with $f\le Cg$,
$f\sim g$ means $f\les g$ and $g\les f$. An assumption of the form $f\ll1$ means that there exists $\eps>0$, typically only
depending on dimension and the H\"older exponent, such that if $f\le\eps$, 
then the conclusion holds.  
We write $|E|$ for the  Lebesgue measure of a set $E$.    
When no confusion is possible, we will drop the integration measures in the integrals. 
For $R>0$ and $x_0\in \R^d$, $B_R(x_0)$ denotes the ball of radius $R$ centered in $x_0$. 
When $x_0=0$, we will simply write $B_R$ for $B_R(0)$. We will also use the notation
\[\dashint_{B_R} f:=\frac{1}{|B_R|}\int_{B_R} f.\]
Let $\rho_0$ and $\rho_1$ be two  densities with compact support in $\R^d$  and equal mass. We say that  $T$ is an optimal transport map between $\rho_0$ and $\rho_1$ if it 
 minimizes  
\begin{equation}\label{prob:OT}
W^2(\rho_0,\rho_1):= \min_{T\sharp \rho_0=\rho_1} \int_{\R^d} |T-x|^2 \rho_0,
\end{equation}
where by a slight abuse of notation $T\sharp \rho_0$ denotes the push-forward by $T$ of the measure $\rho_0 dx$. We refer the reader to \cite{Viltop} for the existence, uniqueness and characterization of such maps.


%
%

\section{Proof of Theorem \ref{theoepsintro}}  

Let $T$ be the minimizer of \eqref{prob:OT}. As in \cite{GO}, the proof of Theorem \ref{theoepsintro} is based on the decay properties of the excess energy
\begin{equation}\label{def:excess}
 \E(\rho_0,\rho_1,T,R):=R^{-2}\dashint_{B_R} |T-x|^2 \rho_0.
\end{equation}
As already alluded to, the main ingredient for the proof of Theorem \ref{theoepsintro} is the following harmonic approximation result (which by scaling we state for $R=1$).

\begin{proposition}\label{prop:Lagestim}
 For every $0<\tau\ll1$, there exists $\eps(\tau,d)>0$ and $C(\tau,d)>0$ such that if
 \begin{equation}\label{hypsmall}
  \E(\rho_0,\rho_1,T,1)+ \|1-\rho_0\|_{L^\infty(B_1)}^2+\|1-\rho_1\|_{L^\infty(B_1)}^2\le \eps,
 \end{equation}
then there exists  a  function $\phi$ harmonic in $B_{1/2}$, such that 
 \begin{equation}\label{eq:estimdistharmLag}
  \int_{B_{1/2}} |T-(x+\nabla \phi)|^2 \rho_0\le \tau \E(\rho_0,\rho_1,T,1)+ C \lt(\|1-\rho_0\|_{L^\infty(B_1)}^2+\|1-\rho_1\|_{L^\infty(B_1)}^2\rt)
 \end{equation}
 and 
 \begin{equation}\label{eq:energieestimphiLag}
  \sup_{B_{1/2}} |\nabla \phi|^2\les \E(\rho_0,\rho_1,T,1)+\|1-\rho_0\|_{L^\infty(B_1)}^2+\|1-\rho_1\|_{L^\infty(B_1)}^2.
 \end{equation}

\end{proposition}
\begin{proof}
 To simplify notation let $\mathcal{E}:= \E(\rho_0,\rho_1,T,1)$ and $D:=\|1-\rho_0\|_{L^\infty(B_1)}^2+\|1-\rho_1\|_{L^\infty(B_1)}^2$. 
 The claim is an almost direct application of \cite[Th. 1.5]{GHO}. Notice first that for $i=0,1$,
 \begin{equation}\label{W2infty}
  W^2\lt(\rho_i\restr B_1, \frac{\rho_i(B_1)}{|B_1|}\chi_{B_1}\rt)+ \lt(\frac{\rho_i(B_1)}{|B_1|}-1\rt)^2\les \|1-\rho_i\|_{L^\infty(B_1)}^2.
 \end{equation}
 Therefore, \cite[Th. 1.5]{GHO} gives the existence of a radius $R\in (3/4,4/5)$, a constant $c\in \R$ and a couple $(\rho,j)$ solving in the distributional
 sense the continuity equation\footnote{note that $(\rho,j)$ is actually the solution of the Eulerian version of \eqref{prob:OT}, see \cite{GHO}.} 
 \begin{equation}\label{conteq}
  \partial_t\rho+ \nabla\cdot j=0 \ \textrm{ on } \R^d\times(0,1) \qquad \textrm{ and } \qquad \rho(\cdot,0)=\rho_0, \ \rho(\cdot,1)=\rho_1
 \end{equation}
such that the following holds. If $\Phi$ solves
\[
 \Delta \Phi= c \ \textrm{ in } B_{R} \qquad \textrm{ and } \qquad \nu \cdot \nabla \Phi=\nu\cdot \int_0^1 j dt \ \textrm{ on } \partial B_{R},
\]
where $\nu$ denotes the external normal to $\partial B_R$, then 
\[
 \int_{B_{1/2}}|T- (x+\nabla \Phi)|^2 \rho_0 \le \tau \E + C D
\]
and 
\[
 \sup_{B_{1/2}} |\nabla \Phi|^2 \les \E+D.
\]
Using \eqref{conteq} and integration by parts, it is readily seen that we must have 
\[
 c=\dashint_{B_R} (\rho_0-\rho_1)
\]
and thus $|c|^2\les D$. Taking $\phi:= \Phi-\frac{c}{2d}|x|^2$ and using triangle inequality we get \eqref{eq:estimdistharmLag} and \eqref{eq:energieestimphiLag}.
\end{proof}
\begin{remark}
 Instead of appealing to \cite[Th. 1.5]{GHO} whose proof is quite long and intricate, one could alternatively
 give a direct proof of Proposition \ref{prop:Lagestim} following almost verbatim
  the proof of  \cite[Prop. 3.5]{GO}. One would only need  to replace the use of \cite[Lem. 2.2]{GO}, which required the densities to be H\"older continuous, by Lemma \ref{lem:Poi} below.
With respect to \eqref{eq:estimdistharmLag}, this would lead to the slightly more quantitative statement
\[
  \int_{B_{1/2}} |T-(x+\nabla \phi)|^2 \rho_0\les \E(\rho_0,\rho_1,T,1)^{\frac{d+2}{d+1}}+\|1-\rho_0\|_{L^\infty(B_1)}^2+\|1-\rho_1\|_{L^\infty(B_1)}^2.
\]
\end{remark}

\begin{lemma}\label{lem:Poi}
For $g\in L^\infty({B}_1)$, every solution  solution $\phi$ of 
\[
  \Delta \phi=g \ \textrm{ in } B_{1} \qquad \textrm{ and } \qquad \nu \cdot \nabla \phi=\frac{1}{\H^{d-1}(\partial B_1)}\int_{B_1} g \ \textrm{ on } \partial B_1, 
\]
satisfies
\[
\sup_{B_{1}}|\nabla\phi|^2  \les \|g\|_{L^\infty(B_1)}^2.
\]
\end{lemma}
\begin{proof}
 This follows from global Schauder estimates \cite[Th. 3.16 (iii)]{Troianello} and the fact that if $g\in L^\infty(B_1)$, then $g$ is in the Morrey space $L^{2,d-2(1-\alpha)}(B_1)$ for every $0<\alpha<1$ (see \cite{Troianello}).
\end{proof}

We now prove that  as in \cite[Prop. 3.6]{GO}, this estimate implies an ``excess improvement by tilting''-estimate. Even though the proof is similar to the one in \cite{GO}, we include it for the reader's convenience.  

\begin{proposition}\label{iter}
 For every $\beta\in (0,1)$ there exist  $\eps(d,\beta)>0$, $\theta=\theta(d,\beta)>0$ and $C_\theta(d,\beta)>0$ with the property that for every $R>0$ such that
 \begin{equation}\label{hyp:Erho}\E(\rho_0,\rho_1,T,R)+\|1-\rho_0\|^2_{L^\infty(B_R)}+\|1-\rho_1\|^2_{L^\infty(B_R)}\le \eps,\end{equation}
 there exist a symmetric matrix $M$ with $\det M=1$ and a vector $b$ with 
\begin{equation}\label{boundBc}|M-Id|^2 +\frac{1}{R^2} |b|^2\les\E(\rho_0,\rho_1,T,R)+\|1-\rho_0\|^2_{L^\infty(B_R)}+\|1-\rho_1\|^2_{L^\infty(B_R)},\end{equation}
such that, letting  $\hat x:=M^{-1} x$,  $\hat y:= M(y-b)$ and then 
\begin{equation}\label{defhat} 
\hat{T}(\hat x):= M(T(x)-b), \quad \hat{\rho}_0(\hat x):= \rho_0(x) \quad  \textrm{ and } \quad \hat \rho_1(\hat y):= \rho_1(y),
\end{equation}
we have 
\begin{equation}\label{improvementE}
 \E(\hat \rho_0,\hat \rho_1,\hat{T},\theta R)\le \theta^{2\beta} \E(\rho_0,\rho_1,T,R)+C_\theta\lt(\|1-\rho_0\|^2_{L^\infty(B_R)}+\|1-\rho_1\|^2_{L^\infty(B_R)}\rt).
\end{equation}
\end{proposition}
\begin{proof}
By a rescaling $\tilde{x}= R^{-1}x$, which amounts to the re-definition $\tilde{T}(\tilde x):=R^{-1} T(R\tilde x)$ (which preserves optimality)
and $\tilde b:= R^{-1} b$, we may assume that $R=1$. \\
As above, we introduce the notation 
\[\E:=\E(\rho_0,\rho_1,T,1) \qquad \textrm{ and  } \qquad D:=\|1-\rho_0\|^2_{L^\infty(B_1)}+\|1-\rho_1\|^2_{L^\infty(B_1)}.\]
Let $\tau\in(0,1)$ to be fixed later and then $\phi$ be the harmonic function given by Proposition \ref{prop:Lagestim}. Define $b:=\nabla \phi(0)$, $A:= \nabla^2\phi(0)$ and set $M:=e^{-A/2}$, so that $\det M=1$. 
 Using \eqref{eq:energieestimphiLag}
from Proposition \ref{prop:Lagestim} and the mean value property for harmonic functions, 
we see that \eqref{boundBc} is satisfied.

 Defining $\hat \rho_i$ and $\hat T$ as in \eqref{defhat}  we have by  (\ref{boundBc}) and \eqref{hyp:Erho}
\begin{align*}
\dashint_{B_{\theta}} |\hat{T}-\hat{x}|^2 \hat \rho_0&= \dashint_{M B_{\theta}} | M (T-b)-M^{-1} x|^2 \rho_0\\
&\les \dashint_{B_{2\theta}} | T-(M^{-2} x+b)|^2\rho_0\\
&\les \dashint_{ B_{2\theta}}|T-(x+\nabla \phi)|^2\rho_0+\dashint_{ B_{2\theta}}|(M^{-2}-Id-A)x|^2\rho_0\\
&\qquad 
+\dashint_{ B_{2\theta}}|\nabla \phi-b-Ax|^2\rho_0\\
&\les \dashint_{ B_{2\theta}}|T-(x+\nabla \phi)|^2\rho_0+ \theta^2|M^{-2}-Id-A|^2
+\sup_{ B_{2\theta}}|\nabla \phi-b-Ax|^2.
\end{align*}
Recalling $M=e^{-A/2}$, $A=\nabla^2\phi(0)$, and $b=\nabla\phi(0)$, we obtain
\begin{align*}
\theta^{-2}\dashint_{B_{\theta}} |\hat{T}-x|^2\hat \rho_0
&\stackrel{\eqref{eq:estimdistharmLag}}{\les} \theta^{-(d+2)}\lt(\tau \E+ C_\tau D\rt) +|\nabla^2 \phi(0)|^4+\theta^2 \sup_{B_{2\theta}} |\nabla^3 \phi|^2\\
&\stackrel{\eqref{eq:energieestimphiLag}}{\les} \theta^{-(d+2)}\lt(\tau \E+ C_\tau D\rt)+  \lt(\E+D\rt)^2+\theta^2\lt(\E+D\rt)\\
&\les \lt( \tau \theta^{-(d+2)} + \theta^2\rt)\E + C_\tau \theta^{-(d+2)}D,
 \end{align*}
 where we used the harmonicity of $\nabla \phi$ and the fact that $\E+D\ll \theta^2$ (recall that $\theta$ has not been fixed yet).
We may thus find a constant $C(d)>0$ such that 
\[
 \theta^{-2}\dashint_{B_\theta} |\hat{T}-x|^2\hat{\rho}_0\le C\lt( \tau \theta^{-(d+2)} + \theta^2\rt)\E + C_\tau \theta^{-(d+2)}D.
\]
 We now fix $\theta(d,\beta)$ such that $C\theta^2 \le \frac{1}{2} \theta^{2\beta}$, which is possible because $\beta<1$. We finally choose $\tau\ll1$ such that also 
 $ C \tau \theta^{-(d+2)}\le \frac{1}{2} \theta^{2\beta}$, which concludes the proof of \eqref{improvementE}. 
\end{proof}

We may finally prove our main $\eps-$regularity result. As already pointed out in the introduction, it is in this iteration argument that the proof departs from the one in \cite{GO}. Indeed,
under the assumption that the densities are merely continuous, the distance to the identity of the linear transformations $M$ in \eqref{boundBc} are not decaying and we need to compensate
the possible blow-up of the cumulated linear transformations by downgrading the $C^{1,\alpha}$ estimates to $C^{0,\alpha}$ estimates.  A similar argument is used in \cite{DePFig}.  
Notice that the Campanato iteration itself is somewhat simpler  here compared to \cite[Prop. 3.7]{GO} since we do not need to introduce an extra dilation factor at every step to propagate
the smallness assumption on the data.   
\begin{theorem}\label{P3}
 For every $\alpha\in(0,1)$, if 
 \begin{equation}\label{hyp:ErhoR}
 \E(\rho_0,\rho_1,T,2R)+\|1-\rho_0\|_{L^\infty(B_{2R})}^2+\|1-\rho_1\|_{L^\infty(B_{2R})}^2\ll 1, 
 \end{equation}
 then $T$ is of class $C^{0,\alpha}$ in $B_{R}$.

\end{theorem}

\begin{proof}
By scale invariance, we may assume that $R=1$. Let us fix $\alpha\in (0,1)$. By Campanato's theory, see \cite[Th. 5.5]{Giaquinta}, we have to prove that \eqref{hyp:ErhoR} implies
\begin{equation}\label{toproveC1alpha}
 \sup_{x_0\in B_1} \sup_{r\le \frac{1}{2}} \min_{b} \frac{1}{r^{2\alpha}} \dashint_{B_r(x_0)}|T-b|^2 \les  1. 
\end{equation}
Let us first notice that \eqref{hyp:ErhoR} implies that for every $x_0\in B_1$
 \begin{equation}\label{Ex0}
  \E:=\dashint_{B_1(x_0)} |T-x|^2\rho_0 \ll 1 \qquad \textrm{and} \qquad \|1-\rho_0\|^2_{L^\infty(B_1(x_0))}+\|1-\rho_1\|^2_{L^\infty(B_1(x_0))}\ll1 .
 \end{equation}
 Therefore, in order to prove \eqref{toproveC1alpha}, it is enough to show that \eqref{Ex0} implies that for $r\le \frac{1}{2}$,
 \begin{equation}\label{toproveEx0}
  \min_{b}  \dashint_{B_r(x_0)}|T-b|^2\les r^{2\alpha}.
 \end{equation}
Without loss of generality we may now assume that $x_0=0$. To simplify notation, we  let
\[
 \eps:= \E+\|1-\rho_0\|^2_{L^\infty(B_1)}+\|1-\rho_1\|^2_{L^\infty(B_1)}.
\]

Fix from now on  $\beta\in(0,1)$ and let $\theta(d,\beta)$ be given by Proposition \ref{iter}. Thanks to \eqref{Ex0}, Proposition \ref{iter} applies and 
there exist a (symmetric) matrix $M_1$ of unit determinant and a vector $b_1$ such that $T_1(x):=  B_1(T(M_1x)-b_1)$, $\rho_0^1(x):=\rho_0(M_1 x)$ and $\rho_1^1(x):=\rho_1(M_1^{-1} x +b_1)$ satisfy
 \begin{equation}\label{estimEk1}
 \E_1:=\E(\rho_0^1,\rho_1^1,T_1,\theta )\le \theta^{2\beta} \E+C_\theta \lt(\|1-\rho_0\|^2_{L^\infty(B_1)}+\|1-\rho_1\|^2_{L^\infty(B_1)}\rt)\le (\theta^{2\beta} +C_\theta)\eps.
 \end{equation}
 If $T$ is a  minimizer of \eqref{prob:OT}, then so is  $T_1$ with $(\rho_0,\rho_1)$ replaced by $(\rho_0^1,\rho_1^1)$. Indeed,  because $\det M_1=1$, $T_1$ sends $\rho_0^1$ on $\rho_1^1$  and if $T$ is the gradient of a convex function $\psi$
 then  $T_1=\nabla \psi_1$ where $\psi_1(x):=\psi(M_1x)-b_1  \cdot M_1 x$ is also a convex function, which characterizes optimality \cite[Th. 2.12]{Viltop}. 
 Moreover, since by \eqref{boundBc}, we have $|M_1-Id|^2\les \eps$ and $|b_1|^2\les \theta^2 \eps$, if $\eps$ is small enough then $ M_1 B_{\theta}\subset B_1$ and $M_1^{-1} B_\theta +b_1 \subset B_1$, 
 \begin{equation}\label{preservboundHolder}
  \|1-\rho_0^1\|_{L^\infty(B_\theta)}^2+\|1-\rho_1^1\|_{L^\infty(B_\theta)}^2\le \|1-\rho_0\|_{L^\infty(B_1)}^2+\|1-\rho_1\|_{L^\infty(B_1)}^2\le \eps.
\end{equation}

 Therefore, we may iterate  Proposition \ref{iter}, $K>1$ times
 to find a sequence of (symmetric) matrices $M_k$ with $\det M_k=1$, 
 a sequence of vectors $b_k$ and a sequence of maps $T_k$ such that setting for $1\le k\le K$,
 \[T_k(x):= M_k(T_{k-1}(M_k x)-b_k), \quad \rho_0^k(x):=\rho_0^{k-1}(M_k x) \quad \textrm{and} \quad \rho_1^{k}(x):=\rho_1^{k-1}(M_k^{-1} x +b_k),\] 
 we have  
 \begin{align}
  \|1-\rho_0^k\|_{L^\infty(B_{\theta^k})}^2+\|1-\rho_1^k\|_{L^\infty(B_{\theta^k})}^2&\le \eps \label{estimrhok}\\
 \E_k:=\E(\rho_0^k,\rho_1^k,T_k,\theta^k)&\le \theta^{2\beta }\E_{k-1}+C_\theta \eps,\label{estimEk}\\
 |M_k-Id|^2&\les\E_{k-1} +\eps, \label{estimBk}\\
 \frac{1}{\theta^{2(k-1)}} |b_k|^2&\les \E_{k-1} +\eps.\label{estimbk}
 \end{align}
A simple induction argument shows that from \eqref{estimEk} we get
\begin{equation}\label{estimEkprime}
 \E_k\le \theta^{2k \beta} \E+C_\theta \sum_{j=0}^{k-1} \theta^{2 \beta j} \eps\les \eps
\end{equation}
and so \eqref{estimBk} and \eqref{estimbk} lead to 
\begin{equation}\label{estimBkprime}
 \max(|M_k|^2, |M_k^{-1}|^2)\le (1+C\sqrt{\eps}) \qquad \textrm{and } \qquad |b_k|^2\les \theta^{2k} \eps.
\end{equation}
In particular, this implies that $M_K B_{\theta^K}\subset B_{\theta^{K-1}}$ and $M_K^{-1} B_{\theta^K} +b_K \subset B_{\theta^{K-1}}$ so that we may keep iterating Proposition \ref{iter}.

 Letting, $A_k:=M_kM_{k-1}\cdots M_1$  and $d_k:= \sum_{i=1}^k  M_k M_{k-1}\cdots M_i b_i$, we see that  $T_k(x)= A_kT( A_k^*x)-d_k$. By \eqref{estimBkprime}, 
\begin{equation}\label{estimAk}
\max(|A_k|^2,|A^{-1}_k|^2)\le  (1+ C\sqrt{\eps})^k.\end{equation}
We first estimate by definition of $T_k$, the fact that $\det A_k=1$ and $\|1-\rho_0^k\|_{L^{\infty}(B_{\theta^k})}\ll 1$,
\begin{align*}
 \dashint_{A_k^*(B_{\theta^k})} |T+A_k^{-1} d_k|^2&\les \dashint_{A_k^*(B_{\theta^k})} |T-A_k^{-1} A_k^{-*}x+A_k^{-1} d_k|^2 +\dashint_{A_k^*(B_{\theta^k})} |A_k^{-1} A_k^{-*}x|^2\\
 &\les\dashint_{B_{\theta^k}} |A_k^{-1} (T_k-x)|^2 +\theta^{2k} |A_k^{-1}|^2\\
 &\les |A_k^{-1}|^2 \lt( \E_k + 1\rt)\theta^{2k}\\
 &\stackrel{\eqref{estimAk}\& \eqref{estimEkprime}}{\les} (1+C\sqrt{\eps})^k \theta^{2k}.
\end{align*}
Now if $\eps$ is small enough,  \eqref{estimAk} yields  $B_{\frac{1}{2}\lt(\frac{\theta}{1+C\sqrt{\eps}}\rt)^k}\subset A_k^* (B_{\theta^k})$ and therefore 

\begin{align*}
 \min_{b}  \dashint_{B_{\frac{1}{2}\lt(\frac{\theta}{1+C\sqrt{\eps}}\rt)^k}}|T-b|^2&\les (1+ C \sqrt{\eps})^{k d} \dashint_{A_k^*(B_{{\theta}^k})}|T+A_k^{-1}d_k|^2\\
 &\les (1+C \sqrt{\eps})^{k(d+2)} \theta^{2k}.
\end{align*}
Since $\alpha$ and $\theta$ are fixed, if $\eps$ is small enough, then $1+C\sqrt{\eps}\le \theta^{-\frac{2(1-\alpha)}{d+2(1+\alpha)}}$ so that 
\[
 \theta^2 (1+C \sqrt{\eps})^{d+2}\le \lt(\frac{\theta}{1+C \sqrt{\eps}}\rt)^{2\alpha}
\]

From this \eqref{toproveEx0} follows, which concludes the proof of \eqref{toproveC1alpha}.
 
\end{proof}

\section*{Acknowledgment}
I warmly thank G. De Philippis for suggesting that our proof from \cite{GO} could cover the case of merely continuous densities.
I also thank A. Zilio for pointing out the reference \cite{Troianello}. Part of this research was supported by the  project 
ANR-18-CE40-0013 SHAPO financed by the French Agence Nationale de la Recherche (ANR) and by the LYSM LIA AMU CNRS ECM INdAM. 
I also thank the Centro de Giorgi in Pisa for its hospitality. 
\bibliographystyle{amsplain}
\bibliography{OT}
 \end{document}